\documentclass[11pt]{amsart}

\usepackage[mathscr]{eucal}
\usepackage{amsmath,amssymb,amscd,amsthm}
\usepackage{color}
\usepackage{epsfig}
\newtheorem{Theorem}{Theorem}
\newtheorem{Lemma}{Lemma}

\newtheorem{Definition}{Definition}
\newtheorem{Corollary}{Corollary}

\theoremstyle{definition}
\newtheorem{Remark}{Remark}

\theoremstyle{plane}

\def \beq{ \begin{equation} }
\def \eeq{\end{equation}}

\title{The Vlasov-Poisson equation, the M\"{o}bius geometry and the $n$-body
    problem in a negative space form}

\begin{document}

\maketitle

\markboth{Ortega Pedro Pablo and  Reyes Victoria J. Guadalupe}{The Vlasov-Poisson equation, M\"{o}bius geometry and the $n$-body problem}

\vspace{-0.5cm}

\author{
\begin{center}
{\rm PEDRO PABLO ORTEGA PALENCIA \\
         Departamento de Matem\'aticas  \\
         Universidad de Cartagena \\
         Cartagena de Indias, \\
         COLOMBIA\\
         {\tt portegap@unicartagena.edu.co}\\
        \medskip
         J. GUADALUPE REYES VICTORIA \\
         Departamento de Matem\'aticas \\
         UAM-Iztapalapa \\
         M\'exico, D.F. \\
         MEXICO \\
         {\tt revg@xanum.uam.mx}}
\end{center}}

 \bigskip
\begin{center}
\today
\end{center}



\begin{abstract}
By using, the Vlasov-Poisson equation defined in either a Riemannian or a semi-Riemannian space $\mathbb{R}^k_g$,  and a Dirac distribution function, 
we re-obtain the well known and classical equations of motion of a mechanical system with a pairwise acting potential function. 
We apply this result to the study of an  $n$--body problem in a two dimensional negative space form  with the hyperbolic cotangent potential. 
Following the  Klein's geometric Erlangen program, with methods of M\"{o}bius geometry and using the Iwasawa decomposition of the M\"{o}bius isometric group $SL(2,\mathbb{R})$ via its representation in one Clifford Algebra, we  complete the study of the whole set of M\"{o}bius solutions (relative equilibria)
of the problem begun by Diacu {\it et al.} in \cite{Diacu8}.
\end{abstract}

\newpage

\section{Introduction}
\label{sec:intro}

The relevance of modern mathematics to physical problems, which require several branches of mathematics for their solution, is ever increasing nowadays.
Among those physical problems, is the problem of generalizing the Newton gravitational equations for spaces with non-zero Gaussian 
curvature as Lobachevsky and Boljai tried in the 19th century (see \cite{Diac}). In recent years, the problem has been revived and a series of papers have been published using suitable potentials that generalize the Newtonian one.  In these, various interesting solutions from the geometric point of view have been obtained.  Even though the methods for obtaining the equations of motion in each curved $n$-body problem has been worked out usging particular methos for each case, and only for constant Gaussian curvature (see \cite{Diac}, \cite{Diacu8}, and \cite{Perez}), all the equations of motion obtained converge, in suitable systems of intrinsic coordinates, to the classical and well known equations (\ref{eq:equations-curved-motions}) (see \cite{Abraham} and \cite{Einstein} ),
which allow to compare the geometry of the space (geodesic curves) with the dynamics of the physical problem (gradient, in the metric, of the potential).

In this paper, by using, the classical Vlasov-Poisson 
equation defined in a space  $\mathbb{R}^k_G$ (Riemannian or semi-Riemannian) and a Dirac distribution function defined on positions and velocities, we re-obtain the equations of motion of the $n$-body problem, where the particles are moving under the action of a pairwise acting  potential. The aim of doing this is to emphasize the hierarchical position of this partial differential equation (see \cite{Dulov}), and to develop in the future suitable methods for studying questions related to the $n$-body problem on an arbitrary manifold for a large number of interacting mass-particles.
 
We apply this result for studying the motion of $n-$ point masses in a negative space form (space of constant Gaussian curvature; see \cite{DoCarmo}) moving under the influence of a gravitational potential that naturally extends Newton's law as Kozlov did in \cite{Kozlov}, and Diacu {\it et al} in 
\cite{Diac} and \cite{DiacuPerez}. This work completes the earlier study of this problem done by Florin Diacu, Ernesto P\'erez-Chavela and J.\ Guadalupe Reyes Victoria in  \cite{Diacu8}. To do this complementation  
with the help of algebraic methods  we formally define the Killing vector fields in $\mathbb H^2_R$ associated
to the true conic motions, through the representation of its isometric group $SL(2,\mathbb{R})$ into a suitable Clifford Algebra.

For an interesting history of this $n$-body problem, the reader is referred to the seminal papers of Florin Diacu, {\it et al}  \cite{Diacu1}, \cite{DiacuPerez}  and   \cite{Diac}. 

\smallskip

The paper is organized as follows.

\smallskip

In section \ref{sec:Vlasov-equamotion}, we recall the Vlasov-Poisson equation (VP) in  a general
non-euclidian space and obtain from it the equations of motion for a general mechanical system with a general pairwise acting potential. We apply this result to directly obtain the equation of motion for the $n-$body problem in $\mathbb{H}^2_R$ with a hyperbolic cotangent potential.

 In section \ref{sec:invariant-Moebius}, following the Klein's Erlangen program (see \cite{Kisil}),
 we work out the representation of the M\"{o}bius group of isometries $SL(2,\mathbb{R})$ of $\mathbb{H}^2_R$ into a four dimensional Clifford algebra. This allows us to factorize such group, via the  Iwasawa decomposition Theorem, 
by one-dimensional subgroups, in three different suitable ways, which define three different geometries: the elliptic, the parabolic and the hyperbolic (see \cite{Kisil}). Next, we obtain from  the Lie algebra $sl(2,\mathbb{R})$, when we project it via the exponential map onto the Lie group $SL(2,\mathbb{R})$,  the possible five Killing vector fields associated to all one-parameter  subgroups, which are the factors of such decomposition. These vector fields  define the M\"{o}bius solutions (also called {\it relative equilibria} in the dynamic literature) of the problem.

In section \ref{sec:Klein_model}, by using the same decomposition, we proceed to find the M\"{o}bius solutions to the problem in $\mathbb{H}^2_R$, by the method of matching the gravitational field with each one of the Killing vector fields obtained in section \ref{sec:invariant-Moebius} as it was done before in \cite{Diacu8}, \cite {Perez} and \cite{Perez-Reyes}. In subsections \ref{subsec:hyperbolic-relative eq}, \ref{subsec:parabolic-relative eq} and
\ref{subsec:elliptic-relative eq} we obtain functional algebraic conditions for the existence of the hyperbolic normal M\"{o}bius solutions, the  parabolic nilpotent M\"{o}bius solutions and the  elliptic cyclic M\"{o}bius solutions respectively. We do the same for the parabolic cyclic M\"{o}bius solutions in subsection \ref{subsec:parabolic-cyclic-ic-relative eq}, and for the hyperbolic cyclic M\"{o}bius solutions in subsection \ref{subsec:hyperbolic-cyclic-ic-relative eq}, and we prove here that, in this problem, there are no solutions to any of these two types.

We note that the non-existence of the aforementioned types of true conic orbits (parabolic cyclic and hyperbolic cyclic) for a two-dimensional negative space form is not a surprise,  given the fact in the parabolic euclidian geometry some geometrical features are richer (see \cite{Kisil}).


\section{Vlasov-Poisson equation and equations of motion for the $n$-body problem}\label{sec:Vlasov-equamotion}

In an euclidian space $\mathbb{R}^k$ with coordinates $x=(x^1,x^2, \cdots, x^k)$, the Vlasov-Poisson equation or equation of self-consistent field has the form (see \cite{Dulov}),
\begin{equation}\label{eq:Vlasov}
\frac{\partial F}{\partial t} + \left<v, \frac{\partial F}{\partial x} \right> +
\left<f(F), \frac{\partial F}{\partial v} \right>= 0,
\end{equation}
where $f$ is a functional of the distribution function $F$ of  particles  moving along the space $\mathbb{R}^k$
with velocities $v=(v^1,v^2, \cdots, v^k)$.

A simple kind of force $f$ for the Vlasov-Poisson equation (\ref{eq:Vlasov}) is given by
\begin{equation}\label{eq:force}
f(x,t)=-\nabla \left(\int \, U(x,y)\, F(y,v,t) \, dv dy \right),
\end{equation}
where $U=U(x,y)$ defines an interactive pairwise potential.

For the non-euclidean space $\mathbb{R}^k_g$  endowed with the Riemannian or Semiriemannian metric $g=(g_{ij})$, which we will call from now on 
simply metric, we consider the corresponding Vlasov-Poisson equation (see \cite{Dulov}),
\begin{equation}\label{eq:curved-Vlasov}
\frac{\partial F}{\partial t} + \left<v, \frac{\partial F}{\partial x} \right>_{g} +
\left<f(F), \frac{\partial F}{\partial v} \right>_{g}= 0,
\end{equation}
where $< \, , \, >_{g}$ denotes the scalar product in the given metric. We denote by $ \nabla_{z}^g $ the gradient with respect to the variable $z$ in the metric, and by $\displaystyle \frac{D}{dt}$  the corresponding covariant derivative.

\smallskip

 For one system of $n$-particles moving along the positions $X^i(t)$ and velocities $V^i(t)$ in the 
$k-$dimensional space $\mathbb{R}^{3}_g$ endowed with the metric $g=(g_{ij})$ we obtain the equations of motion for the system associated to the Vlasov-Poisson equation (\ref{eq:curved-Vlasov}).

\begin{Theorem}\label{Theo:equations-curved-motions} The  solutions of one system of $n$ point particles with weights (masses) $m_1, m_2, m_3, \cdots, m_n$  moving under the influence of the pairwise acting potential $U=U(x,y)$, along the positions $X^i(t)$ and velocities $V^i(t)$ in the  space $\mathbb{R}^k_g$,  satisfy the system of $n$ second order differential equations,
\begin{equation}\label{eq:curved-motion}
\frac{D V^i}{d t} =  -  \sum_{j=1}^n \, m_j \nabla_{x}^g   U(X^i,X^j).
\end{equation}
for $ i=1,\dots,n$.
\end{Theorem}

\begin{proof} We consider, for  the force function (\ref{eq:force}), the particular case when
the distribution function of particles moving under the influence of the potential $U=U(x,y)$ is,
\begin{equation}\label{eq:Dirac-distribution}
F(t,v,x)=\sum_{i=1}^n \, m_i  \delta(v-V^i(t)) \, \delta(x-X^i(t)),
\end{equation}
where $\delta=\delta (x)$ is the ordinary Dirac-function, and $X^i(t)$, $V^i(t)$ are time dependent vector functions
which locate the position and velocity of the $i$-th particle.

\smallskip

The following chain of  results hold.

For a fixed $i = 1,2,3, \cdots$ we have,  by using the properties of the Dirac function in the force,
\begin{eqnarray}\label{eq:force-xi}
f(t,x) &=& - \nabla_{x}^g  \int U(x,y) F(t,v,y) dy \, dv  \nonumber \\
&=& - \nabla_{x}^g  \int U(X^i,y) \sum_{j=1}^n \, m_j  \delta(v-V^j(t)) \, \delta(y-Y^j(t)) dy \, dv  \nonumber \\
&=& -  \sum_{j=1}^n \, m_j \nabla_{x}^g  \int U(X^i,y)   \delta(v-V^j(t)) \, \delta(y-Y^j(t)) dy \, dv  \nonumber \\
&=& -  \sum_{j=1}^n \, m_j \nabla_{x}^g   U(X^i,X^j),   \nonumber \\
\end{eqnarray}
where we have used $x=X^i(t)$, $v=V^i(t)$ and the corresponding integral has value 1.

By a straight forward computation we obtain,
\begin{eqnarray}\label{eq:partial-t}
\frac{\partial F}{\partial t} &=&  \sum_{i=1}^n \, m_i \left<\nabla_{v}^g \delta (v-V^i(t)), -\frac{D V^i }{d t} \right>_g \delta(x-X^i(t))  \nonumber \\
 &+& \sum_{i=1}^n \, m_i \left< \nabla_{x}^g \delta (x-X^i(t)), -\dot{X^i} \right>_g \delta(v-V^i(t)) \nonumber \\
&=&  \sum_{i=1}^n \, m_i \left[\left<\nabla_{v}^g \delta (v-V^i(t)), -\frac{D V^i}{d t} \right>_g
 + \left<\nabla_{x}^g \delta (x-X^i(t)), -\dot{X^i} \right>_g  \right], \nonumber \\
\end{eqnarray}
where we have used again that $x=X^i(t)$, $v=V^i(t)$.

On the other hand,
\begin{eqnarray}\label{eq:product-Fx-v}
\left<v, \frac{\partial F}{\partial x} \right>_g &=& \left<v,
\sum_{i=1}^n \, m_i \delta (v-V^i(t)) \, \nabla_{x}^g \delta (x-X^i(t)) \right>_g  \nonumber \\
&=& \sum_{i=1}^n \, m_i \, \left<v,  \, \nabla_{x}^g \delta (x-X^i(t)) \right>_g,  \nonumber \\
\end{eqnarray}
for the case when $v=V^i(t)$.

Finally, we have that,
\begin{eqnarray}\label{eq:product-Fv-f}
& & \left<f(F), \frac{\partial F}{\partial v} \right>_g    \nonumber \\
 &=& \left< - \sum_{j=1}^n \, m_j \nabla_{x}^g   U(X^i,X^j),
\sum_{i=1}^n \, m_i \delta (x-X^i(t)) \, \nabla_{v}^g \delta (v-V^i(t)) \right>_g \nonumber \\
&=& \sum_{i=1}^n \, m_i \, \left< - \sum_{j=1}^n \, m_j \nabla_{x}^g   U(X^i,X^j),
  \, \nabla_{v}^g \delta (v-V^i(t)) \right>_g \nonumber \\
\end{eqnarray}
for the case when $x=X^i(t)$.

Therefore, when we substitute equations (\ref{eq:partial-t}), (\ref{eq:product-Fx-v}) and (\ref{eq:product-Fv-f}) and factorize $\displaystyle \sum_{i=1}^n \, m_i$ in equation (\ref{eq:curved-Vlasov}), a sufficient condition for this equation to hold is that 
\begin{eqnarray}\label{eq:condition-Vlasov}
0 &=& \left< \nabla_{v}^g \delta (v-V^i(t)), -\frac{D V^i}{d t} \right>_g
 + \left<\nabla_{x}^g \delta (x-X^i(t)), -\dot{X^i} \right>_g \nonumber \\
 &+&  \left<v,  \, \nabla_{x}^g \delta (x-X^i(t)) \right>_g + \left< -  \sum_{j=1}^n \, m_j \nabla_{x}^g   U(X^i,X^j), \, \nabla_{v}^g \delta (v-V^i(t)) \right>_g, \nonumber \\
 \end{eqnarray}
for all $i=1,2,\cdots ,n$.

Compare the first and last terms in the right hand side of equation (\ref{eq:condition-Vlasov}) and
the second and third ones.  If we put $\displaystyle v =\dot{X}^i$ as above then, for such a number to vanish, it is necessary that
\begin{equation}\label{eq:curved-motion-2}
\frac{D V^i}{d t} =  -  \sum_{j=1}^n \, m_j \nabla_{x}^g   U(X^i,X^j),
\end{equation}
which ends the proof.
\end{proof}

 If $x=(x^1, x^2, \cdots , x^k)$  is the system of coordinates for $\mathbb{R}^k_g$, and
$\displaystyle \{ \Gamma^l_{ij} \}$ is the Levi-Civita connection compatible with the metric then, for any pair of vectors $X$ 
and $V=\dot{X}$ in $\mathbb{R}^k_g$, the equalities
 \[ \left(\frac{D V}{d t}\right)^i = \left(\frac{D \dot{X}}{d t}\right)^i = \ddot{x}^i+\Gamma^i_{lj} \dot{x}^l \dot{x}^j, \]
and
 \[ \sum_{j=1}^n \, m_j \nabla_{x}^g   U(X^i,X^j)= m_j g^{ij} \frac{\partial U}{\partial x^j} \]
hold, where $g^{-1}=(g^{ik})$ is the inverse matrix for the metric $g$ and, using the Einstein notation, the summation runs over the upper and lower equal indices.
Thus, for this coordinate system, equations (\ref{eq:curved-motion}) can be written, without loss of generality, as
\begin{equation} \label{eq:equations-curved-motions}
 \ddot{x}^i+\Gamma^i_{lj} \dot{x}^l \dot{x}^j = - m_k g^{ik} \frac{\partial U}{\partial x^k}.
\end{equation}

\begin{Remark}\label{rem:geo-dyn}
We observe that the vanishing of the right hand side of equation
 (\ref{eq:equations-curved-motions}) defines the equations of geodesic curves associated to the metric $(g_{ij})$, whereas the right hand side corresponds to the curved (via the metric) gradient of the potential $U$. In other words, we are comparing geometry versus dynamics, as in the classical way (see \cite{Abraham}
and \cite{Einstein}).
\end{Remark}

\begin{Remark} We observe that if $g_{ij}= \delta_{ij}$ is the euclidian metric, then $\Gamma^i_{lj} \equiv 0$ and we obtain the classical Newtonian $n$-body problem.
\end{Remark}

\begin{Definition} A two dimensional negative space form is a smooth  connected surface with negative constant Gaussian curvature.
\end{Definition}

Since the Minding's Theorem (see \cite{DoCarmo}) states that all surfaces with the same constant Gaussian curvature K are locally isometric,
we apply the equations (\ref{eq:equations-curved-motions}) to obtain  the equations of motion of the $n$-body problem in the upper half complex plane 
\[\mathbb{H}^2_R = \{ w \in \mathbb{C} \, | \, {\rm Im} \,(w) >0 \} \]
endowed with the conformal Riemannian metric 
\begin{equation}\label{met-k}
   -ds^2= \frac{4 R^2}{(w-\bar{w})^2}dw d\bar{w},
\end{equation}
with conformal factor
\begin{equation}\label{met-k2}
   \mu (w, \bar{w})= \frac{4 R^2}{(w-\bar{w})^2},
\end{equation}
and with Gaussian constant curvature $\displaystyle K= - \frac{1}{R}$.

The equations of the geodesic curves are given by (see \cite{DoCarmo})
\begin{equation}\label{eq:geodesic-klein}
\ddot{w} - \frac{2   \dot{w}^2}{ w-\bar{w}} = 0,
\end{equation}
and they are either half circles orthogonal to the real axis $(y=0)$ or half lines perpendicular to it.  

The space $\mathbb{H}^2_R$ endowed with the metric (\ref{met-k}) is called the {\it Klein upper half plane} model for the hyperbolic geometry.
(Here we denote by $d_{kj}$ the hyperbolic distance between the points $w_k$ and $w_j$ in $\mathbb{H}^2_R$.)

Following \cite{Kozlov}, \cite{Diac} and  \cite{Diacu8}, for the $n-$body problem on $\mathbb{H}^2_R$ of a system of $n-$particles with respective
masses $m_1, m_2, \cdots , m_n$ and positions $w_1(t), w_2(t), \cdots , w_n(t)$, we use 
the acting hyperbolic cotangent potential in the coordinates $(w, \bar{w})$, given by
\begin{equation}\label{eq:potesfklein-0}
V_R (w, \bar{w})= \frac{1}{R} \sum_{1 \leq k < j \leq n}^n
m_k m_j \coth_R \left(\frac{d_{kj}}{R}\right).
\end{equation}

A straightforward computation shows that
\begin{equation}\label{eq:potesfklein}
V_R (w, \bar{w})= \frac{1}{R} \sum_{1 \leq k < j \leq n}^n
m_k m_j \frac{(\bar{w}_k+w_k)(\bar{w}_j+w_j)-2(|w_k|^2+|w_j|^2)}{[\Theta_{(k,j)}(w, \bar{w})]^{1/2}},
\end{equation}
where
\begin{equation}\label{eq:singularklein}
\Theta_{(k,j)}(w, \bar{w})
=[(\bar{w}_k+w_k)(\bar{w}_j+w_j)-2(|w_k|^2+|w_j|^2)]^2-(\bar{w}_k-w_k)^2(\bar{w}_j-w_j)^2
\end{equation}
defines the singular set of the problem.

A direct substitution
of equations (\ref{eq:geodesic-klein}) and (\ref{eq:potesfklein}) together  with the conformal factor of the metric 
(\ref{met-k}) in equation  (\ref{eq:equations-curved-motions}) shows that the solutions of the system
defining the $n$-body problem in the negative space form $\mathbb{H}^2_R$ with
acting hyperbolic potential (\ref{eq:potesfklein}), satisfy the system of $n$ second order differential
equations, $ k=1,\dots,n$,
\begin{eqnarray} \label{eq:motionkleingeo}
   \ddot{w}_k - \frac{2   \dot{w}_k^2}{ w_k-\bar{w}_k}
  &=& \frac{2}{\mu(w_k, \bar{w}_k)} \,
  \frac{\partial V_R}{\partial \bar{w}_k} =-\frac{( w_k- \bar{w}_k )^2}{2 R^2} \,
  \frac{\partial V_R}{\partial \bar{w}_k} \nonumber  \\
&=& -\frac{2(w_k-\bar{w}_k)^3}{R} \sum_{\substack{j=1\\ j\ne k}}^n \frac{m_j  (\bar{w}_j-w_j)^2(w_k-w_j)(\bar{w}_j-w_k)}{[\Theta_{(k,j)}(w, \bar{w})]^{3/2}}. \nonumber  \\
\end{eqnarray}

As we have pointed in Remark \ref{rem:geo-dyn}, the left hand side of equation (\ref{eq:motionkleingeo}) defines the equations of geodesic curves in
 $\mathbb{H}^2_R$, whereas the right hand side corresponds to the gradient (in the metric (\ref{met-k}))  of the potential $V_R$.


\section{Invariants of the M\"{o}bius Geometry}\label{sec:invariant-Moebius}

In this section we show the geometric invariants (conic curves) of the M\"{o}bius geometry of $\mathbb{H}^2_R$ according to 
 the Klein's Erlangen Program and with the methodology of representing the isometry M\"{o}bius group into a four dimensional Clifford algebra as  in \cite{Kisil}. These invariants will define five Killing vector fields on the hyperbolic half plane.

\smallskip

Let
\[ {\rm SL}(2,\mathbb{R}) = \{ A \in {\rm GL}(2,\mathbb{R}) \, | \, \det A=1 \},  \]
be the {\it special linear real 2-dimensional group}, which is a 3-dimensional simply connected, smooth real manifold. It is well known (see \cite{Dub})  that the {\it
group of proper isometries} of $\mathbb{H}^2_R$ is the projective quotient
group $\displaystyle {\rm SL}(2,\mathbb{R})/\{\pm I \}$. Every class
\[ A = \left(\begin{array}{cc}
    a  &  b     \\
    c  &  d   \\
    \end{array}\right)  \in {\rm SL}(2,\mathbb{R})/\{\pm I \} \]
also has an associated  unique  M\"obius
transformation  $ f_A : \mathbb{H}^2_R \to \mathbb{H}^2_R$, where
\[f_A (z) = \frac{a z +b}{c z + d}, \]
which satisfies $f_{-A} (z)= f_A (z)$.

\smallskip

We now give the representations of the ${\rm SL}(2,\mathbb{R})$ group in a Clifford algebra with two generators, knowing
that there are three different Clifford algebras $C\ell (e)$, $C\ell (p)$, $C\ell (h)$, corresponding to the elliptic, parabolic and hyperbolic cases.

A Clifford algebra $C\ell (\sigma)$ is a four dimensional linear real space spanned by
$\{1, e_0, e_1, e_0e_1 \}$ with the non-commutative product defined by, (see \cite{Kisil} for more details)
\begin{eqnarray}\label{eq:Clifford}
e_0^2 &=& -1, \nonumber \\
e_1^2 &=& \sigma= \left\{
\begin{array}{cc}
  -1  &  \mbox{for $C\ell (e)$}, \\
 0  &  \mbox{for $C\ell (p)$}, \\
  1   &  \mbox{for $C\ell (h)$} \\
\end{array}\right. \nonumber \\
e_0e_1 &= &-e_1e_0. \nonumber  \\
\end{eqnarray}

For the space $\mathbb{R}^2=\{ue_0 + ve_1 \, | \, u,v \in \mathbb{R} \}$, we denote by $\mathbb{R}^e$,
$\mathbb{R}^p$ or $\mathbb{R}^h$ the corresponding Clifford algebra $\mathbb{R}^{\sigma}$. Therefore, an isomorphic representation of the group
${\rm SL}(2,\mathbb{R})$ with the same product is obtained if we replace the matrix $\displaystyle  \left(\begin{array}{cc}
    a  &  b     \\
    c  &  d   \\
    \end{array}\right)$ by $\displaystyle  \left(\begin{array}{cc}
    a  &  be_0     \\
    -ce_0  &  d   \\
    \end{array}\right)$, which has the advantage of defining the M\"{o}bius transformation
    $\mathbb{R}^{\sigma} \to \mathbb{R}^{\sigma}$, given by (see \cite{Kisil} for more details)
\[\left(\begin{array}{cc}
    a  &  be_0     \\
    -ce_0  &  d   \\
    \end{array}\right): \, ue_0 + ve_1 \to \frac{a(ue_0 + ve_1)+be_0}{-ce_0(ue_0 + ve_1)+d}.  \]

In this representation, by the Iwasawa decomposition Theorem (see \cite{Husemuller},
\cite{Iwa}, \cite{Kisil}), we can factorize ${\rm SL}(2,\mathbb{R})=ANK$ 
so that each matrix factorizes as
\begin{equation}
\left(\begin{array}{cc}
    a  &  be_0     \\
    -ce_0  &  d   \\
    \end{array}\right)= \left(\begin{array}{cc}
    -\alpha  &  0     \\
    0  &  \alpha  \\
    \end{array}\right)\left(\begin{array}{cc}
    1  &  \nu e_0     \\
    0  &  1   \\
    \end{array}\right)\left(\begin{array}{cc}
    \cos \phi  &  e_0 \sin \phi    \\
    e_0 \sin \phi  &  \cos \phi  \\
    \end{array}\right),
\end{equation}
with
\[ \alpha^{-1}= \sqrt{c^2+d^2}, \quad \nu= ac+bd, \quad \text{and} \quad \phi= - \arctan \left(\frac{c}{d} \right). \]
The one dimensional Abelian subgroup $A$ is the normalizer of the nilpotent one dimensional subgroup $N$, and $K$ is a maximal compact subgroup of ${\rm SL}(2,\mathbb{R})$.

\smallskip

In the Lie algebra $sl(2,\mathbb{R})$ of ${\rm SL}(2,\mathbb{R})$ we consider the following suitable set of Killing vector fields,
\[ \left\{
X_1 = \frac{1}{2} \left( \begin{array}{ccc}
    1 & 0 \\
     0 & -1  \\
    \end{array}\right),  \quad
X_2 = \left(\begin{array}{cc}
     0 & 1 \\
    0 & 0  \\
    \end{array}\right), \quad
X_3 = \left(\begin{array}{cc}
    0 & e_0 \\
    -e_0 & 0  \\
    \end{array}\right)
\right\}. \]

If we consider also the exponential map of matrices,
\[ \exp:  sl(2,\mathbb{R}) \to {\rm SL}(2,\mathbb{R}), \]
applied to  the one-parameter  additive subgroups (straight lines) $\{ t X_1 \}$, $ \{ t X_2 \}$,  and $\{ t X_3 \}$, we obtain the following one-dimensional factor subgroups of the Lie group ${\rm SL}(2,\mathbb{R}) $.
\begin{enumerate}
\item[\bf 1.]  The isometric normal (homothetic) subgroup,
\[A= \left\{ \exp (t X_1)=
\left(\begin{array}{cc}
    e^{t/2} & 0 \\
    0  & e^{-t/2}  \\
    \end{array}\right) \right\}
\]
which defines the one-parameter family of acting M\"obius transformations
\begin{equation}\label{eq:Mobius-Klein-1}
 f_{1} (w,t) = e^{t}  w,
 \end{equation}
 and associated to the differential equation
 \begin{equation}\label{eq:Mobius-Klein-1a}
 \dot{w} = w.
 \end{equation}
 The flow of the corresponding vector field is a set of straight lines emanating from the origin, as shown in Figure \ref{normalizer-vector-field}.
 \begin{figure}[h]
\centering
\includegraphics[width=6cm]{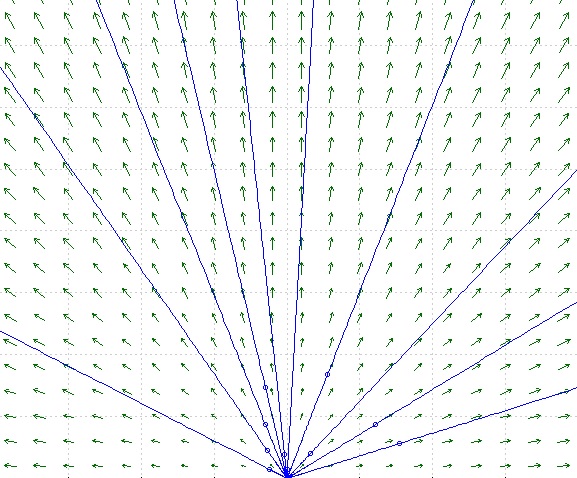}
\caption{The vector field $\dot{w}= w$.}
\label{normalizer-vector-field}
\end{figure}

\item[\bf 2.]  The isometric nilpotent shift subgroup,
\[ N= \left\{\exp (t X_2)=
\left(\begin{array}{cc}
    1 & t \\
    0  & 1 \\
    \end{array}\right) \right\}
\]
which defines the one-parameter family of acting M\"obius transformations
\begin{equation}\label{eq:Mobius-Klein-2}
f_{2} (w,t) =  w + t,
\end{equation}
and associated to the differential equation
 \begin{equation}\label{eq:Mobius-Klein-2a}
 \dot{w} = 1.
 \end{equation}
 The flow of the corresponding vector field is a set of horizontal parallel straight lines, as shown in Figure \ref{nilpotent-vector-field}.
 \begin{figure}[h]
\centering
\includegraphics[width=6cm]{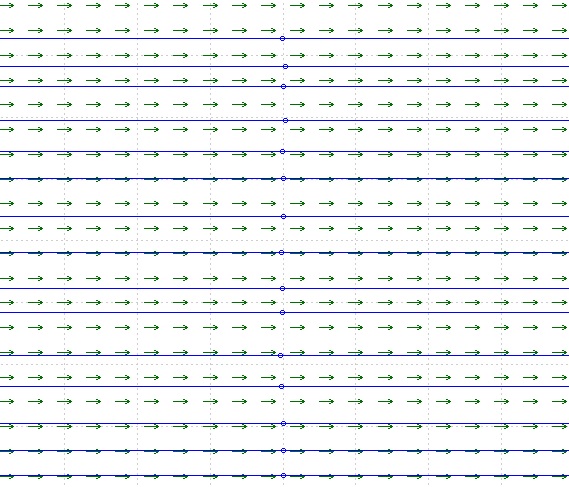}
\caption{The vector field $\dot{w}= 1$.}
\label{nilpotent-vector-field}
\end{figure}

\item[\bf 3.]  The isometric rotation subgroup,
\[ K= \left\{ \exp (t X_3)=
\left(\begin{array}{cc}
    \cos t & e_0 \sin t \\
    e_0 \sin t  & \cos t  \\
    \end{array}\right) \right\}
\]
which defines the one-parameter family of acting M\"obius transformations
\begin{equation}\label{eq:Mobius-Klein-3}
f_{3} (w,t) = \frac{ (\cos t) \, w + e_0 \sin t }{ ( e_0 \sin t) \, w + \cos t},
\end{equation}
and is associated to the differential equation
 \begin{equation}\label{eq:Mobius-Klein-3a}
 \dot{w} = e_0(1-w^2).
 \end{equation}

The three different Killing vector fields associated to $\mathbb{R}^e$, $\mathbb{R}^p$ or $\mathbb{R}^h$ for each corresponding Clifford algebra are as follows.
 \begin{enumerate}
\item[a.]  The flow of the corresponding vector field in $\mathbb{R}^e$ is a set of coaxal circles with focus (in the sense of M\"{o}bius geometry, see \cite{Kisil}) at the point $w=i$, as shown in Figure \ref{elliptic-vector-field}.
 \begin{figure}[h]
\centering
\includegraphics[width=6cm]{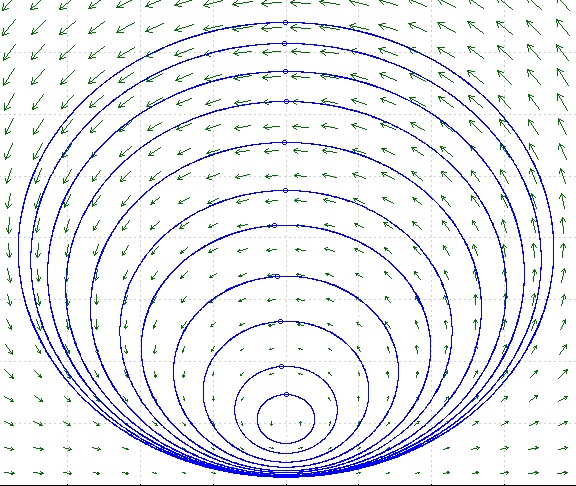}
\caption{The elliptic vector field $\dot{w}= e_0(1-w^2)$ for $e_1^2=-1$.}
\label{elliptic-vector-field}
\end{figure}

\item[b.]  The flow of the corresponding vector field in $\mathbb{R}^p$ is a set of vertical parabolas with
horizontal directrices, as shown in Figure \ref{parabolic-vector-field}.
 \begin{figure}[h]
\centering
\includegraphics[width=6cm]{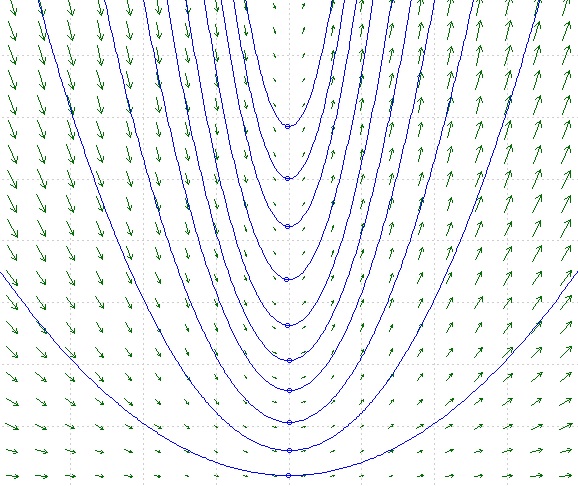}
\caption{The parabolic vector field $\dot{w}= e_0(1-w^2)$ for $e_1^2=0$.}
\label{parabolic-vector-field}
\end{figure}

\item[c.]  The flow of the corresponding vector field in $\mathbb{R}^h$ is a set of vertical hyperbolas with  asymptotes parallel to the diagonal, as shown in Figure \ref{hyperbolic-vector-field}.
 \begin{figure}[h]
\centering
\includegraphics[width=6cm]{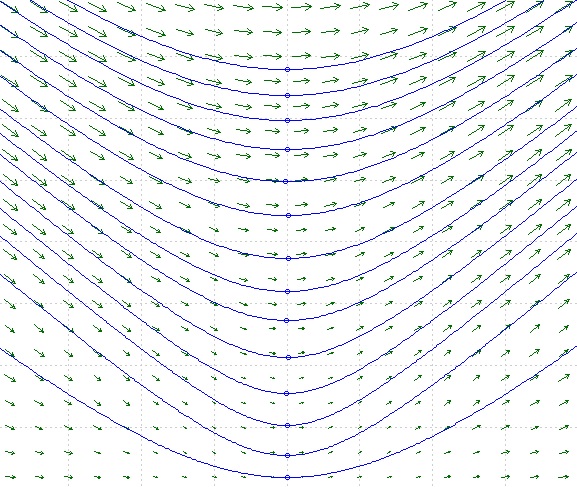}
\caption{The hyperbolic vector field $\dot{w}= e_0(1-w^2)$ for $e_1^2=1$.}
\label{hyperbolic-vector-field}
\end{figure}
\end{enumerate}
\end{enumerate}

Via the exponential map, even though every matrix in the Lie Algebra $sl(2,\mathbb{R})$ defines a one-parametric subgroup of $SL(2,\mathbb{R})$,
there are only three classes of one-parametric subgroups in  $SL(2,\mathbb{R})$ under the conjugation $B \to MBM^{-1}$ (see Kisil \cite{Kisil2}).

\begin{Lemma}\label{prop:principal-subgroups} Any continuous one-parametric subgroup of $SL(2,\mathbb{R})$ is conjugated to $A$,
$N$ or $K$.
\end{Lemma}


\section{M\"{o}bius solutions of the $n$-body problem in $\mathbb{H}^2_R$}
\label{sec:Klein_model}

In this section we obtain the whole set of the so called {\it M\"{o}bius solutions} or {\it relative equilibria solutions} of the mechanical system (\ref{eq:motionkleingeo}). For this, we will use the method of matching the gravitational hyperbolic cotangent  field with each one of the Killing vector fields associated to the above one-dimensional subgroups as  in \cite{Diacu8} and \cite{Perez-Reyes}.

\begin{Definition}\label{def:relative-equilibria}
A  {\it M\"{o}bius solution} or {\it relative equilibrium}  for the $n$--body problem in $\mathbb{H}^2_R$, is a solution $w(t)=(w_1(t), w_2(t), \cdots, w_n(t))$ of the equations of motion (\ref{eq:motionkleingeo}) which is invariant under some one-dimensional subgroup $B(t)$ of the group $SL(2,\mathbb{R})$. In other words, the function obtained by the action denoted by $z(t)= B(t) w(t)$ is also a solution of (\ref{eq:motionkleingeo}).
\end{Definition}

Since, from Lemma \ref{prop:principal-subgroups}, the basic one-dimensional parametric subgroups (\ref{eq:Mobius-Klein-1a}), (\ref{eq:Mobius-Klein-2a}) and (\ref{eq:Mobius-Klein-3a}) define, up to conjugation, the whole set of one-dimensional parametric subgroups in $SL(2,\mathbb{R})$,
we will show in the following subsections the corresponding M\"{o}bius motions associated to such subgroups. From now on, and for short,
we will understand that the whole set of solutions under study are of type M\"{o}bius.

 \subsection{Hyperbolic normal M\"{o}bius solutions}
\label{subsec:hyperbolic-relative eq}

In this subsection we state conditions for the existence of normal hyperbolic solutions
(also called homothetic relative equilibria) for the problem in $\mathbb{H}^2_{R}$.
These conditions were obtained in Diacu {\it et al.} \cite{Diacu8}.

\begin{Definition}
A  {\it normal hyperbolic solution} for the $n$--body problem in $\mathbb{H}^2_R$ is a solution $w(t)=(w_1(t), w_2(t), \cdots, w_n(t))$ of the equations of motion (\ref{eq:motionkleingeo}) which is invariant under the subgroup $A$ associated to the vector field given by (\ref{eq:Mobius-Klein-1a}).
\end{Definition}

A straightforward substitution of equations (\ref{eq:Mobius-Klein-1a}) in equations (\ref{eq:motionkleingeo}) shows that a necessary and sufficient  condition for the function $w=(w_1, \dots, w_n)$ to be a solution of system \eqref{eq:motionkleingeo} and, at the same time, a normal hyperbolic solution is that, for every $k=1,\dots,n,$ the coordinates satisfy the algebraic functional system
\begin{equation} \label{eq:condrationalsystemklein1a}
\frac{R (w_k+\bar{w}_k) \, w_k}{8(w_k- \bar{w}_k)^4} =  \sum_{\substack{j=1\\ j\ne k}}^n
\frac{m_j(w_j-\bar{w}_j)^2(w_k-w_j)(\bar{w}_j-w_k)}{[\Theta_{(k,j)}(w,
\bar{w})]^{3/2}}.
\end{equation}

Therefore, we can also call these solutions {\it hyperbolic normal solutions}.

In \cite{Diacu8} it is proved the existence of such motions  for 2 and 3 interacting particles in $\mathbb{H}^2_{R}$.

\subsection{Parabolic nilpotent M\"{o}bius solutions}
\label{subsec:parabolic-relative eq}

In this subsection we show the  M\"{o}bius solutions associated to the subgroup
generated by the Killing vector field $X_2$ which defines the one-parametric family of acting M\"obius transformations (\ref{eq:Mobius-Klein-2a}) in the upper half plane $\mathbb{H}^2_R$. 

\begin{Definition}
A  {\it Parabolic nilpotent M\"{o}bius solution} for the $n$--body problem in $\mathbb{H}^2_R$ is a solution $w(t)=(w_1(t), w_2(t), \cdots, w_n(t))$ of the equations of motion (\ref{eq:motionkleingeo}) which is invariant under the subgroup $N$ associated to the vector field of equation (\ref{eq:Mobius-Klein-2a}).
\end{Definition}

Again as before, a simple substitution of equations (\ref{eq:Mobius-Klein-2a}) in the system (\ref{eq:motionkleingeo}) shows that a  necessary and sufficient  condition for  the function $w=(w_1,\dots, w_n)$ solution of system
\eqref{eq:motionkleingeo} to be a  Parabolic nilpotent solution, is that the coordinate functions satisfy the algebraic functional 
system of equations 
\begin{equation} \label{eq:condrationalsystemklein2b}
-\frac{R}{4(w_k- \bar{w}_k)^4} = \sum_{\substack{j=1\\ j\ne k}}^n
\frac{m_j
(\bar{w}_j-w_j)^2(w_k-w_j)(\bar{w}_j-w_k)}{[\tilde{\Theta}_{(k,j)}(w,
\bar{w})]^{3/2}}. 
\end{equation}

For this reason, we also call these solutions {\it parabolic nilpotent solutions}.

In \cite{Diacu8} it is proved that in this $n$-body problem there are no parabolic nilpotent solutions.

\subsection{Elliptic cyclic M\"{o}bius solutions}
\label{subsec:elliptic-relative eq}


We now discuss the M\"{o}bius solutions of \eqref{eq:motionkleingeo}  corresponding to the action of the Killing vector field
$ X_3$  associated to the differential equation (\ref{eq:Mobius-Klein-3a}) for the case when $e_1^2=-1$.

It is not hard to see (Kisil \cite{Kisil}) that in this case such equation becomes
\begin{equation}\label{eq:eliptic-killing-vector-field}
 \dot{w}= 1+w^2.
\end{equation}

\begin{Definition}
An  {\it elliptic cyclic solution} for the $n$--body problem in $\mathbb{H}^2_R$ is a solution $w(t)=(w_1(t), w_2(t), \cdots, w_n(t))$ of the equations of motion (\ref{eq:motionkleingeo}) which is invariant under the subgroup $K$ associated to the vector field given by (\ref{eq:eliptic-killing-vector-field}).
\end{Definition}

It is easy to see that a necessary and sufficient condition for a solution of system \eqref{eq:motionkleingeo} $w=(w_1,\dots, w_n)$ to be an elliptic cyclic solution is that for all $k=1,\dots, n,$ the following system of algebraic equations be satisfied
\begin{equation} \label{eq:rationalsystem-1}
\frac{R(1+w_k^2)(1+|w_k|^2)}{ (w_k-\bar{w}_k)^4 }
= \sum_{\substack{j=1\\ j\ne k}}^n \frac{m_j   (\bar{w}_j-w_j)^2(w_k-w_j)(\bar{w}_j-w_k)}{[\Theta_{(k,j)}(w, \bar{w})]^{3/2}}.
\end{equation}

We also call these solutions {\it elliptic cyclic solutions}.
 
 In  \cite{Diacu8} it is proved the existence of this type of solutions for the $n$-body problem in the Poincar\'e disk $\mathbb D_R^2 $ of radius $R$ 
in the complex plane, endowed with the conformal metric
\[ ds^2= \frac{4R^4}{(R^2-|z|^2)^2}dz \, d\bar{z}, \]
for 2 and 3 bodies, when the isometric acting group is $SU(1,1)$. Such  results can be carried to our space via the M\"{o}bius linear fractional transformation $ z : \mathbb H_R^2 \to \mathbb D_R^2 $ 
\begin{equation}\label{eq:Moebius-disk-plane}
z= z(w)= \frac{-Rw+iR^2}{w+iR}.
\end{equation}

For example, if we know that the circle of the elliptic foliation passing through  the points  $\displaystyle  i \, \alpha $ and
$\displaystyle  \frac{i}{\alpha}$ (with $1<\alpha$) is centred at the point $\displaystyle  \frac{i}{2} \left(  \alpha +  \frac{1}{\alpha} \right)$ 
and has radius $\displaystyle  r =\frac{\alpha^2 +1}{2 \alpha}$, which has equation (see \cite{Kisil})
\begin{equation}\label{eq:geometric-coaxal-circle}
2 |w|^2 + i(w - \bar{w}) \, \left(\alpha + \frac{1}{\alpha} \right) + 2 = 0,
\end{equation}
then, by applying the transformation (\ref{eq:Moebius-disk-plane}), we can re obtain the following result for two bodies.

\begin{Corollary} ({\bf Two bodies}) \label{the:elliptic-relat-equi-two-body}
Consider 2 point particles of masses $m_1, m_2>0$ moving with positions $w_1=w_1(t)$ and $w_2=w_2(t)$ respectively in  $\mathbb{H}_R^2$, and along circles (\ref{eq:geometric-coaxal-circle}) whose centres are  on the imaginary axis. Then the function $w=(w_1,w_2)$ is an elliptic cyclic  solution of system \eqref{eq:motionkleingeo} with $n=2$, if and only if, for the circle centred at $\displaystyle  \frac{i}{2} \left(  \alpha +  \frac{1}{\alpha} \right)$ of radius $\displaystyle   =\frac{\alpha^2 +1}{2 \alpha}$
 along which $m_1$ moves, there is a unique circle centred at $\displaystyle \frac{i}{2} \left(\beta+ \frac{1}{\beta} \right)$ of radius $\displaystyle \frac{\beta^2 +1}{2 \beta}$,  along which $m_2$ moves, such that, at every time instant, $m_1$ and $m_2$ are sited on the same geodesic arc of $\mathbb{H}_R^2$ but in opposite sides  of the aforementioned isometric circles. Moreover,
\begin{enumerate}
\item if  $m_2>m_1>0$ and $\alpha$ are given, then $ \beta < \alpha$;
\item if $m_1 =m_2>0$ and $\alpha$ are given, then $\beta = \alpha$;
\item if $m_1>m_2>0$ and $\alpha$ are given, then $\beta >\alpha$.
\end{enumerate}
\end{Corollary}


\subsection{Parabolic cyclic M\"{o}bius solutions}
\label{subsec:parabolic-cyclic-ic-relative eq}

In this subsection we study the M\"{o}bius solutions of \eqref{eq:motionkleingeo}
 corresponding to the action of the Killing vector field
$ X_3$  associated to the differential equation (\ref{eq:Mobius-Klein-3a}) for the case when $e_1^2=0$.                                                                            It is not hard to see that in this case such that equation becomes into the one (see \cite{Kisil}),
\begin{equation}\label{eq:eliptic-killing-vector-field-2}
 \dot{w}= 1+w^2-\frac{(w-\bar{w})^2}{4}.
\end{equation}

\begin{Definition}
A {\it parabolic cyclic solution} for the $n$--body problem in $\mathbb{H}^2_R$, is a solution $w(t)=(w_1(t), w_2(t), \cdots, w_n(t))$ of the equations of motion (\ref{eq:motionkleingeo}) which is invariant under the subgroup $K$ associated to the vector field of equation (\ref{eq:eliptic-killing-vector-field-2}).
\end{Definition}

We obtain the condition for a solution of equation (\ref{eq:motionkleingeo}) to be invariant under
the Killing vector field (\ref{eq:eliptic-killing-vector-field-2}).

\begin{Lemma}\label{thm:existence-2} Consider $n$ point particles
with masses $m_1,\dots, m_n>0$, $n\ge 2$, moving in $\mathbb{H}^2_{R}$ with positions $w_k =w_k(t)$ for $k=1,\dots, n$ respectively. 
A necessary and sufficient condition for  the function $w=(w_1,\dots, w_n)$ to be a parabolic cyclic solution of system \eqref{eq:motionkleingeo}, 
is that for all $k=1,\dots, n,$ the following algebraic functional equations 
\begin{eqnarray} \label{eq:rationalsystem-2}
 &-& \frac{R[(w_k-\bar{w}_k)^2(8-w^2_k+6|w_k|^2+3\bar{w}^2_k)-16(1+w^2_k)(1+|w_k|^2)]}{16 (w_k-\bar{w}_k)^4}
 \nonumber \\
 &=& \sum_{\substack{j=1\\ j\ne k}}^n \frac{m_j
   (\bar{w}_j-w_j)^2(w_k-w_j)(\bar{w}_j-w_k)}{[\Theta_{(k,j)}(w, \bar{w})]^{3/2}} \nonumber \\
\end{eqnarray}
are satisfied.
\end{Lemma}

\begin{proof}
Result follows when we differentiate the equations $\displaystyle \dot{w}_k =  1+w^2_k-\frac{(w_k-\bar{w}_k)^2}{4},$
 for $ k=1,\dots,n$, and
substitute in  the equations of motion \eqref{eq:motionkleingeo}. This  completes the proof.
\end{proof}

 We call also {\it parabolic cyclic solutions} to these solutions of system \eqref{eq:motionkleingeo} satisfying the conditions \eqref{eq:rationalsystem-2}.

We observe that equation (\ref{eq:rationalsystem-2}) is so hard of solve, and instead of trying this, we
 find explicitly the parabolic cyclic flow of (\ref{eq:eliptic-killing-vector-field-2}) and we propose these type of solutions as in \cite{DiacuPerez}.

If we use real coordinates $(u,v)$ for $w= u + i \, v$ , the complex differential equation (\ref{eq:eliptic-killing-vector-field-2}) can be written (see \cite{Kisil}) as the system of real differential equations
\begin{eqnarray}\label{eq:eliptic-killing-vector-field-2-real}
\dot{u} &=& 1+ u^2, \nonumber \\
\dot{v} &=& 2 u v, \nonumber \\
\end{eqnarray}
which when is integrated for the initial conditions $(u(0)= \alpha, v(0)= \beta)$ has the 
isometric parametrization
\begin{eqnarray}\label{eq:sol-eliptic-killing-vector-field-2-real}
u(t) &=& \frac{\alpha +  \tan t}{1- \alpha   \tan t}, \nonumber \\
v(t) &=& \frac{\beta \,  \sec^2 t}{(1- \alpha   \tan t)^2}. \nonumber \\
\end{eqnarray}

In this way, the isometric parametrization of the solution of equation (\ref{eq:eliptic-killing-vector-field-2})
with initial condition $w(0)= \alpha +i \beta$ is 
\begin{equation}\label{eq:complex-sol-eliptic-killing-vector-field-2-real}
w(t)= \frac{\alpha +  \tan t}{1- \alpha   \tan t} + i\, \frac{\beta \,  \sec^2 t}{(1- \alpha   \tan t)^2}.
\end{equation}

Let $s=\tan t$ be an isometric new variable. Then $\sec^2 t =1+s^2$, and
the parametrization (\ref{eq:complex-sol-eliptic-killing-vector-field-2-real}) becomes
\begin{equation}\label{eq:one-complex-sol-eliptic-killing-vector-field-2-real-a}
w(s)= \frac{\alpha +  s}{1- \alpha  s} + i\, \frac{\beta \, (1+s^2) }{(1- \alpha   s)^2}.
\end{equation}
Since
\begin{eqnarray}\label{eq:derivatives-variable-isometric}
\frac{ d s}{d t} &=& 1+s^2, \nonumber \\
\frac{ d^2 s}{d t^2} &=& 2(1+s^2) s, \nonumber \\
\end{eqnarray}
and $\displaystyle w'= \frac{d w}{ d s}$, then
\begin{eqnarray}\label{eq:rule-chain-variable-isometric}
\dot{w} &=& \frac{d w}{d s}\frac{ d s}{d t} = (1+s^2) w', \nonumber \\
\ddot{w} &=& w'' \, \left(\frac{d w}{d s}\right)^2 + w' \, \frac{ d^2 s}{d t^2} = 
w'' (1+s^2)^2 +2s(1+s^2) w'. \nonumber \\
\end{eqnarray} 

If we substitute equations (\ref{eq:one-complex-sol-eliptic-killing-vector-field-2-real-a}), (\ref{eq:derivatives-variable-isometric}) and
(\ref{eq:rule-chain-variable-isometric}) in equations (\ref{eq:rationalsystem-2}) for all $k, j=1,\dots, n,$ the respective real parts become
\begin{eqnarray}\label{eq:real-part}
&&\frac{R (\alpha_k +s)(1+\alpha^2_k)(1-\alpha_k \, s)^2}{64 \beta^2_k (1+s^2)^5} \nonumber \\
&=& \sum_{j \neq k} \frac{m_j \beta_j^2}{[\Theta_{(k,j)}(\alpha, \beta)]^{3/2}(1-\alpha_j \, s)^4} \left(\frac{\alpha_j +s}{1-\alpha_j \, s} - \frac{\alpha_k +s}{1-\alpha_k \, s} \right), \nonumber \\
\end{eqnarray}
whereas the imaginary parts are
\begin{eqnarray}\label{eq:imaginary-part}
&-& \frac{R (1+ \alpha_k^2)[(1+\alpha^2_k)(1-\alpha_k \, s)^2+2](1-\alpha_k \, s)^2}{64 \beta^2_k (1+s^2)^4} \nonumber \\
&=&  \sum_{j \neq k} \frac{m_j \beta_j^2}{[\Theta_{(k,j)}(\alpha, \beta)]^{3/2}(1-\alpha_j \, s)^4} \times \nonumber \\ 
& \times & \left[\left(\frac{\alpha_j +s}{1-\alpha_j \, s} - \frac{\alpha_k +s}{1-\alpha_k \, s} \right)^2 - (1+s^2)
\left(\frac{\beta_j^2}{(1-\alpha_j s)^4}- \frac{\beta_k^2}{(1-\alpha_k s)^4} \right) \right] \nonumber \\
\end{eqnarray}
where
\begin{eqnarray}
\Theta_{(k,j)}(\alpha, \beta) 
&=&  \left[ 4\left(\frac{\alpha_j +s}{1-\alpha_j \, s}\right) \left(\frac{\alpha_k +s}{1-\alpha_k \, s} \right) - 
 2 \,  \Xi_{(k,j)} (\alpha, \beta, s)  \right]^2 \nonumber \\
&-& \frac{16 \beta_k^2 \beta_j^2(1+s^2)^4}{(1-\alpha_k s)^4(1-\alpha_j s)^4} \nonumber \\
\end{eqnarray}
is the corresponding evaluation of the singular part (\ref{eq:singularklein}) in the $2n$-dimensional parametric vector $(\alpha, \beta)$
and
\[ \Xi_{(k,j)} (\alpha, \beta, s)=\left(\frac{\alpha_j +s}{1-\alpha_j \, s}\right)^2+ \left(\frac{\alpha_k +s}{1-\alpha_k \, s} \right)^2
+ \frac{\beta_j^2(1+s^2)^2}{(1-\alpha_j s)^4}+ \frac{\beta_k^2(1+s^2)^2}{(1-\alpha_k s)^4}.  \]

We obtain the following result.

\begin{Theorem}\label{thm:no-existence-2} There are not parabolic cyclic solutions for the 
$n-$body problem in $\mathbb{H}_R^2$.
\end{Theorem}

\begin{proof} Consider $n$ point particles with masses $m_1,\dots, m_n>0$, $n\ge 2$, moving in $\mathbb{H}^2_{R}$ 
under the influence of the hyperbolic cotangent potential with positions $w_k =w_k(t)$ for $k=1,\dots, n$ respectively,
and  satisfying the equations (\ref{eq:real-part}) and (\ref{eq:imaginary-part}). Since the cyclic parabolic flow (\ref{eq:eliptic-killing-vector-field-2})
carries the whole set of solutions to the right positive half plane as $t$ goes to infinity, without  loss of generality we can suppose that the $k$-th particle is one of the last reaching
the imaginary axis. If we take these configuration as initial conditions, then $\alpha_k =0$ and $\alpha_j \geq 0$ for $s=0$. Therefore, for these initial conditions equation (\ref{eq:real-part}) becomes
\begin{equation}\label{eq:real-part-vanish}
0 = \sum_{j \neq k} \frac{m_j \beta_j^2}{[\Theta_{(k,j)}(\alpha, \beta)]^{3/2}} \alpha_j,
\end{equation}
which implies necesarilly that $\alpha_j = 0$ for all $j=1,2, \cdots, n$.

This is, if one particle reaches the imaginary axis for some time, then the whole set of particles are sited also on the imaginary axis
at the same time. This implies that we can suppose, in the early, that all the particles are sited along the imaginary axis, and the isometric 
parametrizations are $w_j= s+ i (1+s^2) \beta_j$ and $w_k= s+ i (1+s^2) \beta_k$ respectively. If this is the case, then, for $s=0$ and
$\alpha_j=0$ in the imaginary parts (\ref{eq:imaginary-part}), we have that
\begin{equation}\label{eq:imaginary-part-contradiction}
 \frac{R}{64 \beta^2_k} = - \sum_{j \neq k} \frac{m_j \beta_j^2}{4[\beta_j^2- \beta_k^2]^2}. 
\end{equation}

The left hand side of equation (\ref{eq:imaginary-part-contradiction}) is 
positive whereas the right hand side is negative. This contradiction proves the claim and ends the proof.
\end{proof}

\subsection{Hyperbolic cyclic M\"{o}bius solutions}
\label{subsec:hyperbolic-cyclic-ic-relative eq}

In this subsection we study the M\"{o}bius solutions of \eqref{eq:motionkleingeo}
corresponding to the action of the Killing vector field
$ X_3$  associated to the differential equation (\ref{eq:Mobius-Klein-3a}) for the case when $e_1^2=1$.
It is not hard to see that in this case such that equation becomes into the one (see \cite{Kisil}),
\begin{equation}\label{eq:eliptic-killing-vector-field-3}
  \dot{w}= 1+w^2-\frac{(w-\bar{w})^2}{2}.
\end{equation}

\begin{Definition}
An  {\it hyperbolic cyclic  solution} for the $n$--body problem in $\mathbb{H}^2_R$, is a solution $w(t)=(w_1(t), w_2(t), \cdots, w_n(t))$ of the equations of motion (\ref{eq:motionkleingeo}) which is invariant under the subgroup $K$ associated to the vector field of equation (\ref{eq:eliptic-killing-vector-field-3}).
\end{Definition}

We obtain the condition for a solution of equation (\ref{eq:motionkleingeo}) to be invariant under
the Killing vector field (\ref{eq:eliptic-killing-vector-field-3}).

\begin{Lemma}\label{thm:existence-3} Consider $n$ point particles
with masses $m_1,\dots, m_n>0$, $n\ge 2$, moving in $\mathbb{H}^2_{R}$
with positions $w_k =w_k(t)$ for $k=1,\dots, n$ respectively. A necessary and sufficient condition for  the
function $w=(w_1,\dots, w_n)$ to be an hyperbolic cyclic solution of system \eqref{eq:motionkleingeo},
is that for all $k=1,\dots, n,$ the following algebraic functional equations 
\begin{eqnarray}
 \label{eq:rationalsystem-4}
&-& \frac{w_k^2+10|w_k|^2+2w_k|w_k|^2+2|w_k|^2\bar{w}_k^2-2\bar{w}_k^4+2w_k^4-3\bar{w}_k^2+4}{ 2(w_k-\bar{w}_k) } \nonumber \\
&=& \sum_{\substack{j=1\\ j\ne k}}^n \frac{m_j
   (\bar{w}_j-w_j)^2(w_k-w_j)(\bar{w}_j-w_k)}{[\Theta_{(k,j)}(w, \bar{w})]^{3/2}} \nonumber \\
\end{eqnarray}
are satisfied.
\end{Lemma}

\begin{proof}
Result follows when we differentiate the equations $\displaystyle \dot{w}_k = 1+w^2_k-\frac{(w_k-\bar{w}_k)^2}{2},$ for
$ k=1,\dots,n$,  and
substitute in  the equations of motion \eqref{eq:motionkleingeo}. This completes the proof.
\end{proof}

Once again, we call {\it hyperbolic cyclic solutions} to  these solutions of system \eqref{eq:motionkleingeo} 
satisfying the system \eqref{eq:rationalsystem-4}.

As in the cyclic parabolic case, we observe that equation (\ref{eq:rationalsystem-4}) is also hard of solving, and again as there,
instead of trying this, we will find explicitly the solutions of the cyclic hyperbolic flow (\ref{eq:eliptic-killing-vector-field-3}) and we will propose these type of solutions.

In this case, for the real coordinates $(u,v)$ in $w= u + i \, v$, the complex differential equation (\ref{eq:eliptic-killing-vector-field-3}) can be written (see \cite{Kisil}) as the system of real differential equations
\begin{eqnarray}\label{eq:eliptic-killing-vector-field-2-real}
\dot{u} &=& 1+ u^2+ v^2, \nonumber \\
\dot{v} &=& 2 u v, \nonumber \\
\end{eqnarray}
which when is integrated for the initial conditions 
\begin{eqnarray}\label{initial-conditions-hyperbolic}
 u(0)+ v(0) &=&  \alpha, \nonumber \\
 u(0)-v(0)  &=&  \beta,   \nonumber \\
 \end{eqnarray}
has the  isometric parametrization
\begin{eqnarray}\label{eq:sol-eliptic-killing-vector-field-2-real}
u(t) &=& \frac{1}{2}\left[\frac{\tan t+\alpha}{1- \alpha \tan t}+ \frac{\tan t+ \beta}{1-\beta \tan t}\right], \nonumber \\
v(t) &=& \frac{1}{2}\left[\frac{\tan t+\alpha}{1-\alpha \tan t}- \frac{\tan t+ \beta}{1- \beta\tan t}\right]. \nonumber \\
\end{eqnarray}

In this way, by using equations (\ref{initial-conditions-hyperbolic}), the isometric parametrization of the solution of equation (\ref{eq:eliptic-killing-vector-field-3})
with initial condition $w(0)= u(0) +i v(0)$ is 
\begin{equation}\label{eq:complex-sol-eliptic-killing-vector-field-3-real}
w(t)= \frac{1}{2}\left[\frac{\tan t+\alpha}{1- \alpha \tan t}+ \frac{\tan t+ \beta}{1-\beta \tan t}\right]+
 \frac{i}{2}\left[\frac{\tan t+\alpha}{1-\alpha \tan t}- \frac{\tan t+ \beta}{1- \beta\tan t}\right].
\end{equation}

Let $s=\tan t$ be again the isometric new variable as before. Then 
the parametrization (\ref{eq:complex-sol-eliptic-killing-vector-field-3-real}) becomes
\begin{equation}\label{eq:one-complex-sol-eliptic-killing-vector-field-3-real-a}
w(s)= \frac{1}{2}\left[\frac{s+\alpha}{1-\alpha s}+ \frac{s+ \beta}{1-\beta s}+
i \left(\frac{s+\alpha}{1-\alpha s}- \frac{s+\beta}{1-\beta s}\right)\right].
\end{equation}

For the arbitrary index $l=1,2, \cdots, n$, we define the equalities,
\begin{eqnarray}\label{eq:letters}
A_l &=& A_l(s)= \frac{1}{2} \frac{\alpha_l +s}{1-\alpha_l s} \nonumber \\
B_l &=& B_l(s)= \frac{1}{2} \frac{\beta_l +s}{1-\beta_l s} \nonumber \\
C_l &=& C_l(s)= A_l+B_l \nonumber \\
D_l &=& D_l(s)= A_l- B_l, \nonumber \\
\end{eqnarray} 
and if we substitute equations (\ref{eq:one-complex-sol-eliptic-killing-vector-field-3-real-a}), (\ref{eq:derivatives-variable-isometric}) and (\ref{eq:rule-chain-variable-isometric}) together with the equations (\ref{eq:letters})  in equations (\ref{eq:rationalsystem-4}), 
for all $k, j=1,\dots, n,$ the respective real parts for $l=k$ become
\begin{eqnarray}\label{eq:real-part-hyperbolic}
&&(1+s^2)^2\left(\frac{\alpha_k(1+\alpha_k^2)}{(1-\alpha_k s)^3} + \frac{\beta_k(1+\beta_k^2)}{(1-\beta_k s)^3} \right) \nonumber \\
&+& s (1+s^2)\left(\frac{1+\alpha_k^2}{(1-\alpha_k s)^2} + \frac{\beta_k(1+\beta_k^2)}{(1-\beta_k s)^2} \right) \nonumber \\
&-& 4 \left[ \frac{(\alpha_k + \beta) +2 (1 -(\alpha_k + \beta_k))s - (\alpha_k + \beta_k)s^2}{(1+s^2)(1-\alpha_k s)(1-\beta_k s)} \right] \times  \nonumber \\
&\times &  \left[\frac{1+\alpha_k^2}{(1-\alpha_k s)^2} + \frac{1+\beta_k^2}{(1-\beta_k s)^2} \right] \nonumber \\
&=& \frac{4 (1+s^2)^3(\alpha_k - \beta_k)^3}{R} \sum_{j \neq k}  \left(\frac{\alpha_j +s}{1-\alpha_j \, s} - \frac{\beta_j +s}{1-\beta_j \, s} \right)^2 \, \frac{m_j D_k (C_j-C_k)}{[\Theta_{(k,j)}(\alpha, \beta)]^{3/2}}. \nonumber \\
\end{eqnarray}
On the other hand, the corresponding imaginary parts are
\begin{eqnarray}\label{eq:imaginary-part-hyperbolic}
&&(1+s^2)^2\left(\frac{\alpha_k(1+\alpha_k^2)}{(1-\alpha_k s)^3} - \frac{\beta_k(1+\beta_k^2)}{(1-\beta_k s)^3} \right)  \nonumber \\
&+& s (1+s^2)\left(\frac{1+\alpha_k^2}{(1-\alpha_k s)^2} - \frac{\beta_k(1+\beta_k^2)}{(1-\beta_k s)^2} \right) \nonumber \\
&+& \frac{8(1+\alpha_k^2)(1+\beta_k^2)}{(\alpha_k- \beta_k)(1+s^2)(1-\alpha_k s)(1-\beta_k s)}  \nonumber \\
&=& - \frac{2 (1+s^2)^3(\alpha_k - \beta_k)^3}{R} \sum_{j \neq k}  \left(\frac{\alpha_j +s}{1-\alpha_j \, s} - 
\frac{\beta_j +s}{1-\beta_j \, s} \right)^2 \times \nonumber \\ 
&\times& \frac{m_j [D_k^2 -D_j^2- (C_j-C_k)^2]}{[\Theta_{(k,j)}(\alpha, \beta)]^{3/2}}. \nonumber \\
\end{eqnarray}
Here, $ \Theta_{(k,j)}(\alpha, \beta)$ is again the corresponding evaluation  of the singular function (\ref{eq:singularklein}) in 
the $2n$--parametric vector $(\alpha, \beta)$.

\smallskip

We obtain the following result.

\begin{Theorem}\label{thm:no-existence-2} There are not hyperbolic cyclic solutions for the 
$n-$body problem in $\mathbb{H}_R^2$.
\end{Theorem}

\begin{proof} The proof follows the same method as for Theorem \ref{thm:existence-2}. Regardless we give it here.
Consider again $n$ point particles with masses $m_1,\dots, m_n>0$, $n\ge 2$, moving in $\mathbb{H}^2_{R}$ under 
the influence of the hyperbolic cotangent potential and  satisfying the equations (\ref{eq:real-part-hyperbolic}) and (\ref{eq:imaginary-part-hyperbolic}). Since the cyclic hyperbolic flow also carries all these solutions to the right positive half plane as $t$ goes to infinity, without  loss of generality we can suppose that the $k$-th particle is one of the last reaching the imaginary axis, this is, $u_k (0) =0$. If we take these configuration as initial conditions, then  
for $s=0$ we have the values
\begin{eqnarray}\label{eq:values-hyperbolic}
A_k(0) &=& \frac{\alpha_k}{2}, \nonumber \\
B_k(0) &=& \frac{\beta_k}{2}, \nonumber \\
C_k(0) &=& \frac{\alpha_k + \beta_k}{2} = u_k (0) =0, \nonumber \\
D_k(0) &=& \frac{\alpha_k - \beta_k}{2} = v_k (0) = -\beta_k >0, \nonumber \\
\end{eqnarray}
which implies that $0 = C_k \leq C_j(0) = A_j(0) + B_j (0)$ for all $j \neq k$. 

Therefore, for these initial conditions the real equation (\ref{eq:real-part-hyperbolic}) becomes
\begin{equation}\label{eq:real-part-hyperbolic-vanishes}
0 = \sum_{j \neq k}  \left(\frac{\alpha_j +s}{1-\alpha_j \, s} - \frac{\beta_j +s}{1-\beta_j \, s} \right)^2 \, \frac{m_j }{[\Theta_{(k,j)}(\alpha, \beta)]^{3/2}} (-\beta_k) C_j,
\end{equation}
which implies that $C_j(0) = A_j(0) + B_j (0) = u_j(0)=0$ for all $j \neq k$, and therefore, if the $k$--particle
reaches the imaginary axis for some time, the whole set of particles also must reach the imaginary axis for this same time.
This also implies that $\alpha_j = - \beta_j >0$ for all $j=1,2, \cdots, n$.

On the other hand, if we consider again for $s=0$ the initial conditions $w_k(0) = i \beta_k$, $w_j(0) = i \beta_j$, 
such that $\beta_j \leq \beta_k$, then $C_j=C_k=0$ and $D_j \leq D_k$. If we substitute in the imaginary 
parts (\ref{eq:imaginary-part-hyperbolic}), then it follows that
\begin{eqnarray}\label{eq:imaginary-part-hyperbolic-contradiction}
&&(\alpha_k-\beta_k)(1+\beta_k^2)+ \frac{8(1+\alpha_k^2)(1+\beta_k^2)}{(\alpha_k- \beta_k)} \nonumber  \\
&=& - \frac{2(\alpha_k - \beta_k)^3}{R} \sum_{j \neq k}  (\alpha_j - \beta_j)^2 \, \frac{m_j[D_k^2 -D_j^2]}{[\Theta_{(k,j)}(\alpha, \beta)]^{3/2}}, \nonumber  \\
\end{eqnarray}
which is a contradiction, because again, the left hand side of this last equation is positive whereas that the right hand side is negative.

This contradiction proves the claim and ends the proof.
\end{proof}


\subsection{Conclusions}

Summarizing the results of this section, we obtain that the unique M\"{o}bius solutions for the $n-$body problem in 
$\mathbb{H}_R^2$ are either elliptic cyclic, normal hyperbolic or the compositions of these (loxodromic). 
This shows that for the whole study of the M\"{o}bius solutions, it should be sufficient considering the Cartan-Haussdorf decomposition $SL(2, \mathbb{R}) = KAK$, which uses only the subgroups A and K. Such last decomposition is connected with the geometry of the complex unit disk (see \cite{Fulton}).

In words of the classical mechanics (see \cite{DiacuPerez}) by Minding's Theorem, this can be paraphrased as:

\begin{Theorem}
The unique relative equilibria for the curved $n-$body problem in a two dimensional  space of negative constant Gaussian curvature are: the elliptic, homothetic or loxodromic solutions. 
\end{Theorem}

\subsection*{Acknowledgments}
The authors acknowledge to professors, Alexander (Sasha) Sinitsyn from Universidad Nacional de Colombia for his  helpful comments throughout the preparation of this work, and Antonio Hern\'andez from UAM-Iztapalapa, M\'exico, for his  help in correcting the literary part of this paper. On the other hand, the second author J. Guadalupe Reyes Victoria acknowledges also the total support received from the UNICARTAGENA, Colombia,  for the realization of this paper.


\end{document}